\documentclass[12pt]{amsart}
\usepackage{baskervald,amssymb,amscd,graphicx,caption,subcaption,amsthm,amsmath,array,pdfsync,tikz, tikz-cd,color,url,float,enumerate,stmaryrd,geometry}
\usepackage[all]{xy}

\usetikzlibrary{calc}

\DeclareMathAlphabet{\mathpzc}{OT1}{pzc}{m}{it}
\makeatletter
\def\blfootnote{\xdef\@thefnmark{}\@footnotetext}
\makeatother

\title{The tangent space of the punctual Hilbert scheme\vspace{-2ex}}
\author{Dori Bejleri and David Stapleton\vspace{-2ex}}

\usepackage[pagebackref]{hyperref}
\hypersetup{
  colorlinks= true,
  linkcolor=[RGB]{140,49,0},
  urlcolor=[RGB]{140,49,0},
  citecolor=[RGB]{140,49,0}
}

\newtheorem{theorem}{Theorem}
\newtheorem{corollary}[theorem]{Corollary}
\newtheorem{lemma}[theorem]{Lemma}
\newtheorem{proposition}[theorem]{Proposition}

\newcommand{\hr}[2]{\hyperref[#1]{#2}}

\newtheorem*{maintheorem}{Theorem A}
\newtheorem*{maincomputation}{Theorem B}

\theoremstyle{definition}
\newtheorem{remark}[theorem]{Remark}

\newtheorem{definition}[theorem]{Definition}

\newcommand{\mb}[1]{\mathbb{#1}}

\newcommand{\mc}[1]{\mathcal{#1}}

\newcommand{\Hom}{\operatorname{Hom}}
\newcommand{\Tr}{\operatorname{Tr}}

\newcommand{\ra}{\rightarrow}
\newcommand{\im}{\operatorname{im}}

\def\AA{{\mathbb A}}
\def\NN{{\mathbb N}}
\def\CC{{\mathbb C}}

\def\PP{{\mathbb P}}

\def\fm{{\mathfrak m}}

\def\SymG{{\mathfrak{S}_n}}
\def\Symtwo{{\mathfrak{S}_2}}

\def\Hns{{S^{[n]}}}
\def\Hnc{{(\CC^2)^{[n]}}}
\def\Sns{{S^{(n)}}}
\def\Snc{{(\CC^2)^{(n)}}}
\def\ocn{{(\Oc_{\CC^2})^{[n]}}}
\def\Enl{{\Ec^{[n]}}}
\def\tnl{{(T_S)^{[n]}}}
\def\tnc{{(T_{\CC^2})^{[n]}}}
\def\Zcaln{{\mathcal{Z}_n}}
\def\Homc{{\mathcal{H}om}}
\def\cotanc{{T^*\PP^1}}
\def\aone{{\CC^2/(\pm 1)}}
\def\derc{{\mathrm{Der}_\CC(-\mathrm{log}E)}}
\def\dern{{\mathrm{Der}_\CC(-\mathrm{log}B_n)}}
\def\derp{{\mathrm{Der}_\CC(-\mathrm{log}E')}}
\def\ddx{{\frac{\delta}{\delta x}}}
\def\ddy{{\frac{\delta}{\delta y}}}

\def\dim{{\mathrm{dim}}}
\def\rk{{\mathrm{rank}}}
\def\crk{{\mathrm{corank}}}

\def\Oc{{\mathcal O}}
\def\Ec{{\mathcal E}}
\def\Fc{{\mathcal F}}

\begin{document}
\maketitle

\section*{Introduction}
\thispagestyle{empty}

The purpose of this paper is to study the Zariski tangent space of the punctual Hilbert scheme parametrizing subschemes of a smooth surface which are supported at a single point. We give a lower bound on the dimension of the tangent space in general and show the bound is sharp for subschemes of the affine plane cut out by monomials.

Let S be a smooth connected complex surface, and denote by $\Hns$ the Hilbert scheme parametrizing length $n$ subschemes of $S$. Fogarty \cite{fogarty} showed $\Hns$ is smooth and irreducible. We write $\Sns$ for the symmetric power of $S$. The Hilbert-Chow morphism:
$$
h:\Hns \ra \Sns,
$$ which sends a length $n$ subscheme to its cycle, is invaluable in the study of $\Hns$. We denote by $P_n$ the \emph{nth punctual Hilbert scheme} which is the reduced fiber of $h$ over a multiplicity $n$ cycle in $\Sns$. Thus $P_n$ parametrizes length $n$ subschemes supported at 1 point. Note that $P_n$ is the same for any smooth surface so throughout we assume $S \cong \CC^2$.

The Hilbert-Chow morphism and the punctual Hilbert scheme have attracted a great deal of attention. Beauville ~\cite{Beauville} has shown that if $S$ is a K3 surface then $\Hns$ is a projective holomorphic symplectic variety (one of few known examples). Mukai ~\cite{Mukai} gave a description of the symplectic form in terms of the pairing on $\operatorname{Ext}^1(I,I)$. For general surfaces, $h$ gives a crepant resolution. Brian\c{c}on ~\cite{Briancon} has shown that $P_n$ is irreducible, and  Haiman ~\cite{haiman} has shown that $P_n$ is the scheme-theoretic fiber of $h$ and that $P_n$ is actually a local complete intersection scheme. The Betti numbers of the punctual Hilbert scheme were computed by Ellingsrud and Stromme ~\cite{EllingStromme}. Iarrobino \cite{Iarrobino} showed the Hilbert scheme of length $n$ subschemes of $\AA^k$ is reducible when $k\ge 3$ and $n$ is large, and Erman \cite{Erman} showed that these Hilbert schemes can acquire arbitrary singularities. Huibregtse ~\cite{Huibregtse1,Huibregtse2} studied questions of irreducibility and smoothness of a variety related to $P_n$ which consists of subschemes of $\Hns$ whose sum in the Albanese variety of $S$ is constant.

There is a natural tautological vector bundle $\tnl$ on $\Hns$ whose fiber at a point corresponding to the length $n$ subscheme $\xi \subset S$ is the $2n$-dimensional vector space $H^0(S,T_S|_\xi)$. In \cite{Stapleton}, the second author showed there is a natural injection of sheaves
$$
\alpha_n: \tnl \ra T_\Hns
$$
and that $\tnl$ is the log-tangent sheaf of the exceptional divisor of $h$. Thus it is natural to expect the degeneracy loci of $\alpha_n$ are connected to the singularities of the exceptional divisor of $h$. To make this precise we relate the rank of $\alpha_n$ to the dimension of the Zariski tangent space of the punctual Hilbert scheme.

\begin{maintheorem}\label{MT}
If $\xi\subset \CC^2$ is a length $n$ subscheme supported at the origin, then $$\dim(T_{[\xi]}P_n) \ge 2n - \rk(\alpha_n|_{[\xi]}) = \crk(\alpha_n|_{[\xi]}).$$ Moreover equality holds when the ideal of $\xi$ is generated by monomials.
\end{maintheorem}

When $\xi$ is a monomial subscheme, the ideal of $\xi$ (written $I_\xi \subset \CC[x,y]$) has an associated Young diagram $\mu_{\xi} \subset \NN^2$ defined as
$$
\mu_{\xi} := \{ (i,j) \in \NN^2 | x^i y^j \notin I_\xi \}.
$$
For example when $I_\xi = (y^4,x^2y^2,x^3y,x^7)$, the length of $\xi$ is 14 and we associate to $\xi$ the following Young diagram:

\begin{center}
\includegraphics[height=4.5cm]{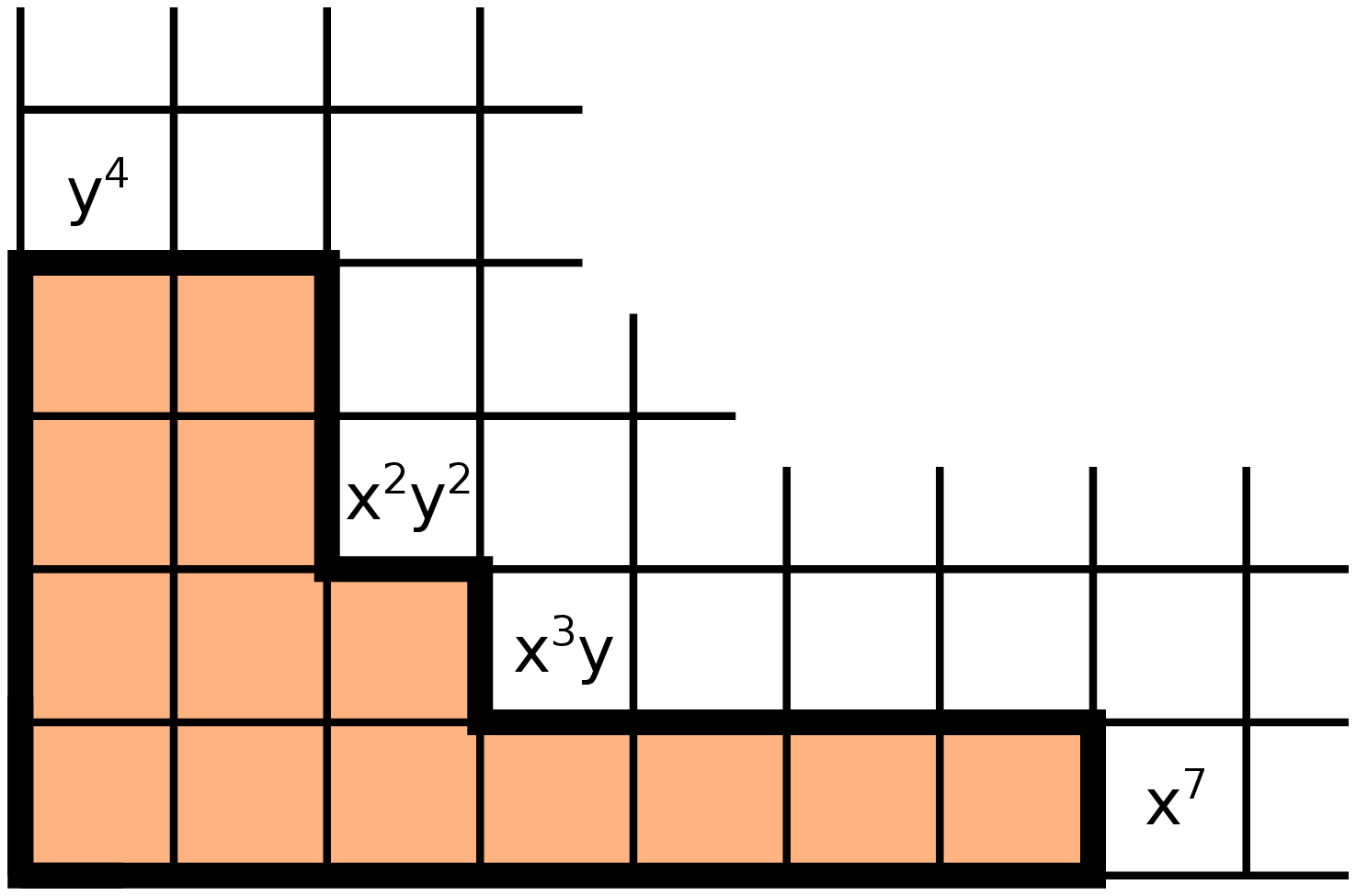}
\end{center}

An elementary statistic associated to $\mu_{\xi}$ is given by tracing the top perimeter of the Young diagram from the top left to the bottom right and keeping track of the horizontal and vertical steps. For example in the above figure we have a sequence of horizontal steps $\Delta h = (2,1,4)$ and vertical steps $\Delta v =(2,1,1)$.

\begin{maincomputation}\label{MC} If $\xi$ is defined by monomials, and $\mu_\xi$ is the corresponding Young diagram then
$$\rk(\alpha_n|_{[\xi]})=\Big( \begin{array}{c}
    \textrm{maximum of horizontal}
  \\
    \textrm{steps of }\mu_\xi
    \end{array} \Big) + \Big( \begin{array}{c}
    \textrm{maximum of vertical}
  \\
    \textrm{steps of }\mu_\xi
    \end{array} \Big).
$$
\end{maincomputation}

\noindent In our example, we have $\rk(\alpha_n|_{[\xi]})=4+2=6$, so $\dim T_{[\xi]}P_n = 28-6=22$.

To prove the inequality in Theorem A we remark that the cokernel of the derivative
$$
dh: \Omega_\Snc \ra \Omega_\Hnc
$$
restricted to $[\xi] \in P_n$ is the cotangent space of $P_n$. This follows from Haiman's result that $P_n$ is the scheme-theoretic fiber of $h$. Moreover $\Hnc$ is equipped with a holomorphic symplectic form \cite[\S 1.4]{nakajima} which gives an isomorphism $\omega: T_\Hnc \cong \Omega_\Hnc$. So to prove the inequality, it suffices to show there is a map:
$$
i:h^*\Omega_\Snc \ra \tnc
$$
such that $dh = \omega \circ \alpha_n \circ i$. The derivative $dh$ factors through $h^*\Omega_\Snc^{\vee \vee}$, a reflexive sheaf, so it is enough to define $i$ away from codimension 2. Away from codimension 2 the map $h$ is \'etale locally a product of the resolution of an $A_1$ singularity with a smooth variety. So the inequality follows after a computation in the case of an $A_1$ singularity, using the interpretation of $\tnl$ as the log-tangent sheaf of the exceptional divisor of $h$.

To carry out the computation in Theorem B we note that the rank of $\alpha_n$ at some $[\xi] \in \Hnc$ is the rank of the map:
$$
H^0(\CC^2,T_{\CC^2}|_\xi) \ra \Hom(I_\xi,\Oc_\xi)
$$
in the normal sequence of $\xi \subset \CC^2$ so we can carry out the computation on $\CC^2$.

To show that equality holds in Theorem A for subschemes of $\CC^2$ cut out by monomials our main computational tool is the affine chart that Haiman introduced in ~\cite{haiman} for $\Hnc$ and the description Haiman gave of the cotangent space at monomial subschemes. Using these tools we explicitly compute the rank of $dh$ at points in $\Hnc$ corresponding to monomial subschemes and show that for $\xi \subset \CC^2$ cut out by monomials:
$$
\rk(dh|_{[\xi]}) =\Big( \begin{array}{c}
    \textrm{maximum of horizontal}
  \\
    \textrm{steps of }\mu_\xi
    \end{array} \Big) + \Big( \begin{array}{c}
    \textrm{maximum of vertical}
  \\
    \textrm{steps of }\mu_\xi
    \end{array} \Big).
$$

We would like to thank our respective advisors Dan Abramovich and Robert Lazarsfeld for their advice and encouragement throughout this project. We are also grateful for conversations with Shamil Asgarli, Kenneth Ascher, Aaron Bertram, Mark de Cataldo, Johan de Jong, Lawrence Ein, Eugene Gorsky, Tony Iarrobino, Daniel Litt, Diane Maclagan, Mark McLean, Luca Migliorini, Mircea Musta\c{t}\v{a}, Hiraku Nakajima, John Ottem, Giulia Sacc\`a, David Speyer, and Zili Zhang.


\section{The proof of the inequality in Theorem A}

In this section we prove the inequality in \hr{MT}{Theorem A}. We start by recalling the main properties of the Hilbert scheme of points that we will need. Let $\Zcaln$ be the universal family of the Hilbert scheme of points on $S$, then $\Zcaln$ has 2 natural projections:

\begin{center}
\begin{tikzpicture}
  \node (Zn) {$S \times \Hns \supset \Zcaln$};
  \node (S) at (3.5,0) {$S$.};
  \node (HnS) at (.95,-1.5) {$\Hns$};
  \draw[->] (1.2,0) to node[above] {$p_1$} (S);
  \draw[->] (.95,-.3) to node [left] {$p_2$} (HnS);
\end{tikzpicture}
\end{center}

\noindent If $\Ec$ is a vector bundle on $S$, then \emph{the tautological bundle associated to} $\Ec$ is $\Enl:= p_{2*}p_1^*\Ec$. The map $\alpha_n$ is obtained by looking at the normal sequence of the inclusion $\Zcaln \subset S\times \Hns$, 
$$
0 \ra T_{\Zcaln} \ra T_{S \times \Hns}|_\Zcaln \cong p_1^* (T_S) \oplus p_2^* (T_\Hns) \xrightarrow{\beta_n} \Homc(I_\Zcaln/I_\Zcaln^2,\Oc_\Zcaln).
$$ 
Applying $p_{2*}(-)$, we see that $p_1^* (T_S)$ pushes forward to $\tnl$ and $\Homc(I_\Zcaln/I_\Zcaln^2,\Oc_\Zcaln)$ pushes forward to $T_{\Hns}$. Then
$$
\alpha_n:=p_{2*}(\beta_n|_{p_1^*(T_S)}).
$$

The symmetric power $\Snc$ is the quotient of $(\CC^2)^{n}$ by the permutation action of the symmetric group on $n$ elements: $\SymG$. The Hilbert-Chow morphism:
$$
h: \Hnc \ra \Snc
$$
maps a point corresponding to a subscheme $[\xi]$ to the $n$-cycle:
$$
h([\xi])=\sum_{p \in \mathrm{Supp}(\xi)}\mathrm{length}_{\CC}(\Oc_{\xi,p})\cdot [p].
$$
The exceptional divisor of $h$, denoted by $B_n$, consists of non-reduced subschemes.

\begin{remark}\label{difseq}
It is always true that for a map of schemes $f: X \ra Y$, if $p \in Y$ and $q \in f^{-1}(p)$ the scheme-theoretic fiber then
$$
T_q f^{-1}(p)\cong Coker(df|_q: f^*\Omega_Y|_q \ra \Omega_X|_q)^{\vee}.
$$
We want to compute the dimension of the Zariski tangent space of $P_n$. As Haiman showed \cite[Prop. 2.10]{haiman} the variety $P_n$ is the scheme-theoretic fiber of $h$, it suffices to compute the corank of $dh$. In particular, the inequality in Theorem A is equivalent to 
$$
\crk(dh|_{[\xi]}) \geq \crk(\alpha_n|_{[\xi]}).
$$ 
\end{remark}

Recall there is a holomorphic symplectic form symplectic form $\omega_n \in H^0(\Hnc,\wedge^2 \Omega_\Hnc)$ on $\Hnc$ \cite[\S 1.4]{nakajima}
which gives an isomorphism $\omega_n:T_\Hnc \cong \Omega_\Hnc$. To bound the corank of $dh$, it suffices to prove that the map $dh$ factors through $\omega_n \circ \alpha_n$:
\begin{center}
\begin{tikzpicture}
  \node (T) {$\tnc$};
  \node (Os) at (-3,-1.5) {$h^*\Omega_\Snc$};
  \node (Oh) at (0,-1.5) {$\Omega_\Hnc.$};
  \draw[dashed,->] (Os) to node[above] {$\exists$} (T);
  \draw[->] (T) to node[right] {$\omega_n \circ \alpha_n$} (Oh);
  \draw[->] (Os) to node[above] {$dh$} (Oh);
\end{tikzpicture}
\end{center}
\noindent As $\omega_n \circ \alpha_n$ is injective is suffices to show that $\omega_n \circ \alpha_n \tnc$ contains the image of $dh$. To do this, we need the following lemma.

\begin{lemma}\label{reflexive}
Suppose $X$ is a normal variety and $\Fc_1,\Fc_2 \subset \Ec$ are subsheaves of a torsion free sheaf on $X$. If $\Fc_2$ is reflexive then the following are equivalent:
\begin{enumerate}
\item $\Fc_1 \subset \Fc_2$ as subsheaves of $\Ec$.
\item There is an open subset $V\subset X$ with codimension 2 complement such that $\Fc_1|_V \subset \Fc_2|_V$ as subsheaves of $\Ec|_U$.
\item There is an \'etale neighborhood $i:U \ra X$ such that the complement of $i(U) = V$ has codimension 2 and $i^* \Fc_1 \subset i^*\Fc_2$ as subsheaves of $i^*\Ec$.
\end{enumerate}
\end{lemma}

\begin{proof}
It is clear $(1)$ implies $(2)$. Now we show the reverse. We remark that $\Fc_1$ is torsion-free, so it includes into its reflexive hull $\Fc_1 \hookrightarrow \Fc_1^{\vee \vee}$. The inclusion $\Fc_1|_V \subset \Fc_2|_V$ as submodules of $\Ec$ extends to an inclusion $\Fc_1|_V^{\vee \vee} \subset \Fc_2|_V$ as any map to a reflexive module factors through the reflexive hull. Now an inclusion of reflexive modules on a normal variety outside codimension 2 uniquely extends to an inclusion on the whole variety. This follows from the fact that reflexive sheaves are \emph{normal} (see \cite[p. 76]{OSS} for the smooth case). Thus we have an inclusion $\Fc_1^{\vee \vee} \subset \Fc_2$. Therefore we have $\Fc_1 \subset \Fc_1^{\vee \vee} \subset \Fc_2$ as submodules of $\Ec$.

Flatness of $i$ proves $(2)$ implies $(3)$. For the reverse, faithful flatness of $i$ mapping onto $V$ gives an inclusion $\Fc_1|_V \subset \Fc_2|_V$ and $V$ is an open set whose complement has codimension $2$. 
\end{proof}

Let $\Symtwo = \langle(12)\rangle \le \SymG$ be the subgroup which exchanges the first 2 elements. Denote by $\Delta \subset (\CC^2)^n$ the big diagonal fixed by $\Symtwo$. The quotient map $\sigma: (\CC^2)^n \ra \Snc$ factors as:
\begin{center}
\begin{tikzpicture}
  \node (C2n) {$(\CC^2)^n$};
  \node (modS2) at (0,-1.5) {$(\CC^2)^n/\Symtwo$};
  \node (Sym) at (3,0) {$\Snc.$};
  \draw[->] (C2n) to node[above] {$\sigma$} (Sym);
  \draw[->] (C2n) to node[left] {$\tau$} (modS2);
  \draw[->] (modS2) to node[below] {$j$} (Sym);
\end{tikzpicture}
\end{center}
\noindent After appropriate change of coordinates $(\CC^2)^n/\Symtwo \cong \aone \times (\CC^2)^{n-1}$, and the symplectic form on the smooth locus of $\Snc$ pulls back and extends to the product symplectic form on the smooth locus of $\aone \times (\CC^2)^{n-1}$. Recall that $\aone \times (\CC^2)^{n-1}$ admits a symplectic resolution:
$$
h_0\times id_{(\CC^2)^{n-1}}: \cotanc \times (\CC^2)^{n-1} \ra \aone \times (\CC^2)^{n-1}
$$
by blowing up $\tau(\Delta)$. Here $\cotanc$ denotes the cotangent bundle of $\PP^1$ with the standard symplectic structure $\omega_\cotanc$ and
$$
h_0: \cotanc \ra \aone
$$
is the minimal resolution of the $A_1$ surface singularity with exceptional divisor $E$.

Let $V$ denote the unramified locus of $j$. $V$ is the complement of the image of all the big diagonals except $\Delta$. Important to us is that $j(V)$ contains all points in $\Snc$ where at most 2 points come together. The following lemma says that if $i:U\ra \Hnc$ is the base change of the \'etale neighborhood $V$ along $h$ then $U$ satisfies the conditions of Lemma 2(3) so that we can check the inclusion $\im(dh) \subset \omega_n \circ \alpha_n(\tnl)$ by pulling back to $U$.

\begin{lemma}\label{etale}
The fiber product:
\begin{center}
\begin{tikzpicture}
  \node (U) {$U$};
  \node (V) at (3,0) {$V$};
  \node (Hilbn) at (0,-1.5) {$\Hnc$};
  \node (A) at (3,-1.5) {$\Snc$};
  \draw[->] (U) to node[above] {$h'$} (V);
  \draw[->] (U) to node[left] {$i$} (Hilbn);
  \draw[->] (V) to node[left] {$j$} (A);
  \draw[->] (Hilbn) to node[above] {$h$} (A);
\end{tikzpicture}
\end{center}
\noindent satisfies $i$ is \'etale and the complement of $i(U)$ has codimension 2. Moreover $U \subset \cotanc \times (\CC^2)^{n-1}$ such that $h_0\times id_{(\CC^2)^{n-1}}|_U = h'$ and the restriction of the symplectic form from $\cotanc \times (\CC^2)^{n-1}$ equals $i^*(\omega_n)$.
\end{lemma}

\begin{proof}
This is essentially the proof that $\Hnc$ admits a holomorphic symplectic form and we refer the interested reader to ~\cite[p. 766]{Beauville} or \cite[\S 1.4]{nakajima}.
\end{proof}

In \cite[Theorem B]{Stapleton} the second author proved the map $\alpha_n$ induces an isomorphism of $\tnl$ with the subsheaf $\dern$ which consists of vector fields tangent to $B_n$. To set up the proof of the inequality in Theorem A we consider the symplectic resolution of the $A_1$-singularity and prove the log-tangent sheaf $\derc$ is isomorphic to the image $dh_0$ as subsheaves of $\Omega_\cotanc$.

\begin{lemma}\label{aone}
The symplectic isomorphism $\omega_\cotanc: T_{\cotanc} \cong \Omega_{\cotanc}$ restricts to an isomorphism of subsheaves:
$$
\omega_\cotanc|_{\derc}: \derc \ra dh_0(h_0^*\Omega_{\aone}).
$$
\end{lemma}

\begin{proof} We have two exact sequences:
\begin{center}
\begin{tikzpicture}
  \node (Tangent) {$T_\cotanc$};
  \node (Derc) at (-3,0) {$\derc$};
  \node (Norm) at (2,0) {$\Oc_E(E)$};
  \node (Omega) at (0,-1.5) {$\Omega_\cotanc$};
  \node (OmegaE) at (2,-1.5) {$\Omega_E$};
  \node (dh) at (-3,-1.5) {$ dh_0(h_0^*\Omega_\aone)$};
  \node (01) at (-6,0) {$0$};
  \node (02) at (4,0) {$0$};
  \node (03) at (-6,-1.5) {$0$};
  \node (04) at (4,-1.5) {$0.$};
  \draw[->] (Tangent) to (Norm);
  \draw[->] (Omega) to (OmegaE);
  \draw[->] (Derc) to (Tangent);
  \draw[->] (dh) to (Omega);
  \draw[->] (Tangent) to node[left] {$\omega_\cotanc$} (Omega);
  \draw[->] (01) to (Derc);
  \draw[->] (03) to (dh);
  \draw[->] (Norm) to (02);
  \draw[->] (OmegaE) to (04);
  \draw[dashed,->] (Derc) to (dh);
  \draw[dashed,->] (Norm) to (OmegaE);
\end{tikzpicture}
\end{center}
If $v\in \derc(U)$ is any logarithmic vector field then $v$ is tangent to $E$, so for any point $p\in E$ with $v|_p \ne 0$ we know $v|_p$ generates the tangent space of $E$. On the other hand the pairing of the 1-form $\omega_\cotanc(v)|_E$ with $v|_p$ vanishes by skew symmetry of $\omega_\cotanc$. So the restricted 1-form $\omega_\cotanc(v)|_E$ vanishes identically. Thus we can fill in the dashed arrows to obtain a commuting diagram with an injection on the left and a surjection on the right. But the surjection on the right is an isomorphism as these are isomorphic line bundles on $E$, thus $\omega_\cotanc|_\derc$ is also an isomorphism.
\end{proof}

\begin{proof}[Proof of inequality in \hr{MT}{Theorem A}]
According to Remark \ref{difseq}, it is enough to show that as subsheaves of $\Omega_\Hnc$ we have the containment $dh(h^*\Omega_\Snc) \subset \omega_n \circ \alpha_n(\tnc)$. If $i: U \ra \Hnc$ is the \'etale open set from Lemma \ref{etale}, then by Lemma \ref{reflexive} it is enough to show that $i^*(dh(h^*\Omega_\Snc)) \subset i^*(\omega_n \circ \alpha_n(\tnl))$ as subsheaves of $i^* \Omega_\Hnc = \Omega_U$.

Let $E'$ denote the exceptional divisor of $h'$. By Lemma \ref{etale}, we have a fiber square with $i$ \'etale and $i^{-1}(B_n) = E'$. It follows that
$$
i^*(dh(h^*\Omega_\Snc)) = dh'(h'^*\Omega_V),\hspace{5mm}\text{and}\hspace{5mm}i^*(\tnc) = \derp.
$$
For the second equality we use the interpretation of $\alpha_n(\tnc)$ as the log-tangent sheaf of $B_n$ \cite[Theorem B]{Stapleton}. On the one hand, the exceptional divisor $E' = U \cap (E\times (\CC^2)^{n-1})$ is locally a product, so the log-tangent sheaf of $E'$ splits as a direct sum:
$$
\derp =(p^*\derc \oplus q^* T_{(\CC^2)^{n-1}})|_U
$$
where $p$ and $q$ denote projection of $\cotanc \times (\CC^2)^{n-1}$ onto $\cotanc$ and $(\CC^2)^{n-1}$ respectively. On the other hand, $h' = (h_0\times id_{(\CC^2)^{n-1}})|_{U}$ so the subsheaf $dh'(h'^*\Omega_V)$ splits as a direct sum:
$$
dh'(h'^*\Omega_V) = (p^*dh_0(h_0^*\Omega_{\aone}) \oplus q^*\Omega_{(\CC^2)^{n-1}})|_U \subset (p^*\Omega_\cotanc \oplus q^*\Omega_{(\CC^2)^{n-1}})|_U = \Omega_U.
$$

Finally by Lemma \ref{etale}, $i^* \omega_n$ is the same as the restriction of the product symplectic form on $\cotanc \times (\CC^2)^{n-1}$. Therefore it suffices to check that the symplectic form on $\cotanc \times (\CC^2)^{n-1}$ identifies $p^*dh_0(h_0^*\Omega_{\aone}) \oplus q^*\Omega_{(\CC^2)^{n-1}}$ with $p^*\derc \oplus q^* T_{(\CC^2)^{n-1}}$. As $i^*\omega_n$ is a product symplectic form it respects this direct sum decomposition. The second factors are clearly identified and the first factors are identified by Lemma \ref{aone}.
\end{proof}

\begin{remark}
The above proof actually shows there is an isomorphism:
$$
h^*(\Omega_{\Snc})^{\vee \vee} \cong \tnc,
$$
that is $\tnc$ is the reflexive hull of $h^*(\Omega_\Snc)$.
\end{remark}


\section{Computing the rank of $\alpha_n$ at monomial subschemes}

In this section we show how to compute the rank of $\alpha_n$ at monomial subschemes, proving \hr{MC}{Theorem B}. During the proof, we exhibit the computation on an example subscheme $\xi \subset \CC^2$ with $I_\xi = (y^4,xy^2,x^2y,x^5)$.

\begin{proof}[Proof of \hr{MC}{Theorem B}] Let $\xi \subset \CC^2$ be a length $n$ subscheme whose ideal $I_\xi$ is defined by monomials. As in the introduction we associate to $\xi$ the Young diagram (see Figure a) $\mu=\mu_\xi \subset \NN^2$ defined as
$$
\mu := \{ (i,j) \in \NN^2 | x^i y^j \notin I_\xi \}.
$$
We associate to $\mu$ the elementary statistic given by tracing the top perimeter of $\mu$ from the top left to the bottom right and recording the horizontal steps $\Delta h$ and the vertical steps $\Delta v$ (see Figure b).

\begin{figure}[H]
    \begin{subfigure}{0.35\textwidth}
        \centering
        \includegraphics[width=.8\textwidth]{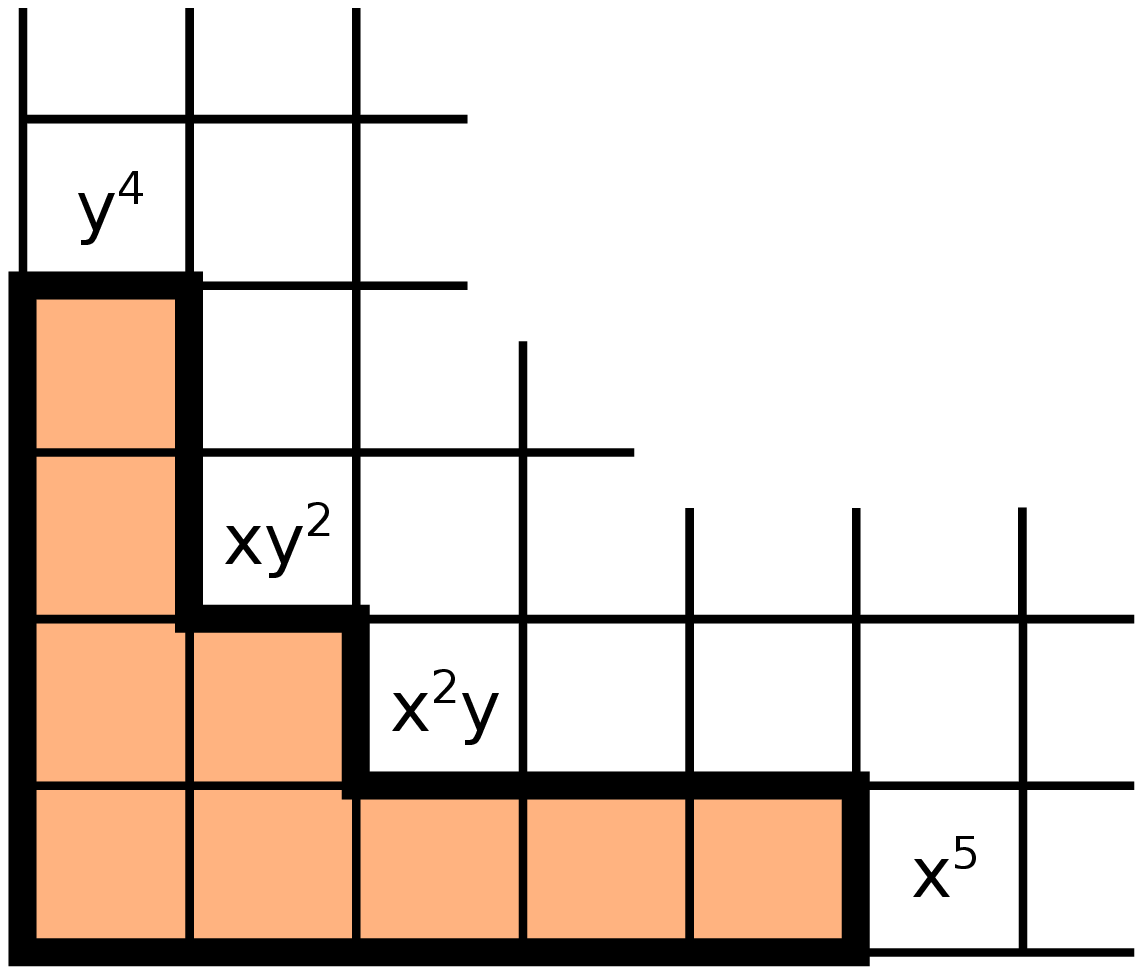}
        \caption*{Figure a. The Young diagram associated to our example $\xi \subset \CC^2$.}
    \end{subfigure}
    \hspace{1cm}
    \begin{subfigure}{0.35\textwidth}
        \centering
        \includegraphics[width =.8\textwidth]{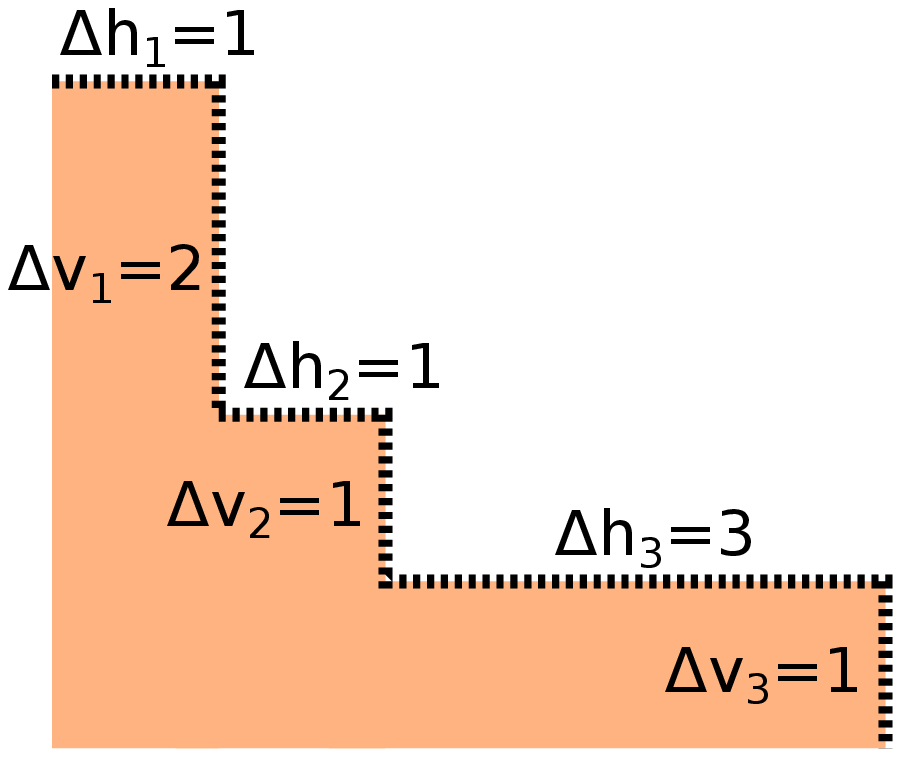}
        \caption*{Figure b. In our example $\xi$ we have $\Delta h = (1,1,3)$ and $\Delta v = (2,1,1)$.}
    \end{subfigure}
\end{figure}

Our aim is to compute $\rk(\alpha_n|_{[\xi]})$. There are natural isomorphisms: $\tnc|_{\xi} \cong H^0(T_{\CC^2}|_{\xi}) \cong \CC[\xi]\ddx \oplus \CC[\xi]\ddy$ and $T_\Hnc|_{[\xi]} \cong \Hom(I_\xi,\CC[\xi])$. Moreover the map $\alpha_n|_\xi$ is exactly the map in the normal sequence associated to $\xi \subset \CC^2$ which maps any restricted derivation $\delta \in T_{\CC^2}$ to a homomorphism by:

\begin{center}
$\displaystyle \begin{array}{lr} \alpha_n|_{[\xi]}: \CC[\xi] \ddx \oplus \CC[\xi] \ddy \to \Hom(I_\xi, \CC[\xi]), & \delta \mapsto \Big( \begin{array}{c}
    I_{\xi} \xrightarrow{\alpha_n|_{[\xi]}(\delta)} \CC[\xi] \\
    f \mapsto \delta(f)|_{\xi}
  \end{array} \Big) \end{array}$.
\end{center}

As $I_\xi$ is generated by monomials we can decompose $I_\xi = \bigoplus_{\NN^2 \setminus \mu} \CC\cdot x^iy^j$ as a $\CC$-vector space. Moreover, the ring of functions on $\xi$ has a monomial $\CC$-vector space basis $\CC[\xi] = \bigoplus_{\mu} \CC \cdot x^iy^j$. Observe that for any \emph{monomial} derivation $x^iy^j \ddx$ or $x^iy^j \ddy$ the associated homomorphism in $\Hom(I_\xi,\CC[\xi])$ maps our basis of $I_\xi$ to our basis of $\CC[\xi]$ up to possible scaling. This makes it possible to understand these homomorphisms combinatorially. For example $\ddy$ acts up to scaling by decreasing the power of $y$ by 1, which on $\NN^2$ is a shift down operator, annihilating any $(i,j)$ of the form $(i,0)$ (see Figure c). The derivation $y \ddx$ acts by shifting left by 1 and shifting up by 1 (see Figure d).

More importantly, all monomials are eigenvectors for $x\ddx$ and $y\ddy$. Therefore, the homomorphisms associated to $x\ddx$ and $y\ddy$ are 0 in $\Hom(I_\xi,\CC[\xi])$, and any multiple $x^{i+1}y^j\ddx$ or $x^iy^{j+1}\ddy$ for $i,j>0$ is 0 as a homomorphism $I_\xi \to \CC[\xi]$. So the only possible nonzero homomorphisms coming from monomial derivations are of the form $y^j\ddx$ or $x^i\ddy$.

\begin{figure}
    \begin{subfigure}{0.3\textwidth}
        \centering
        \includegraphics[width=\textwidth]{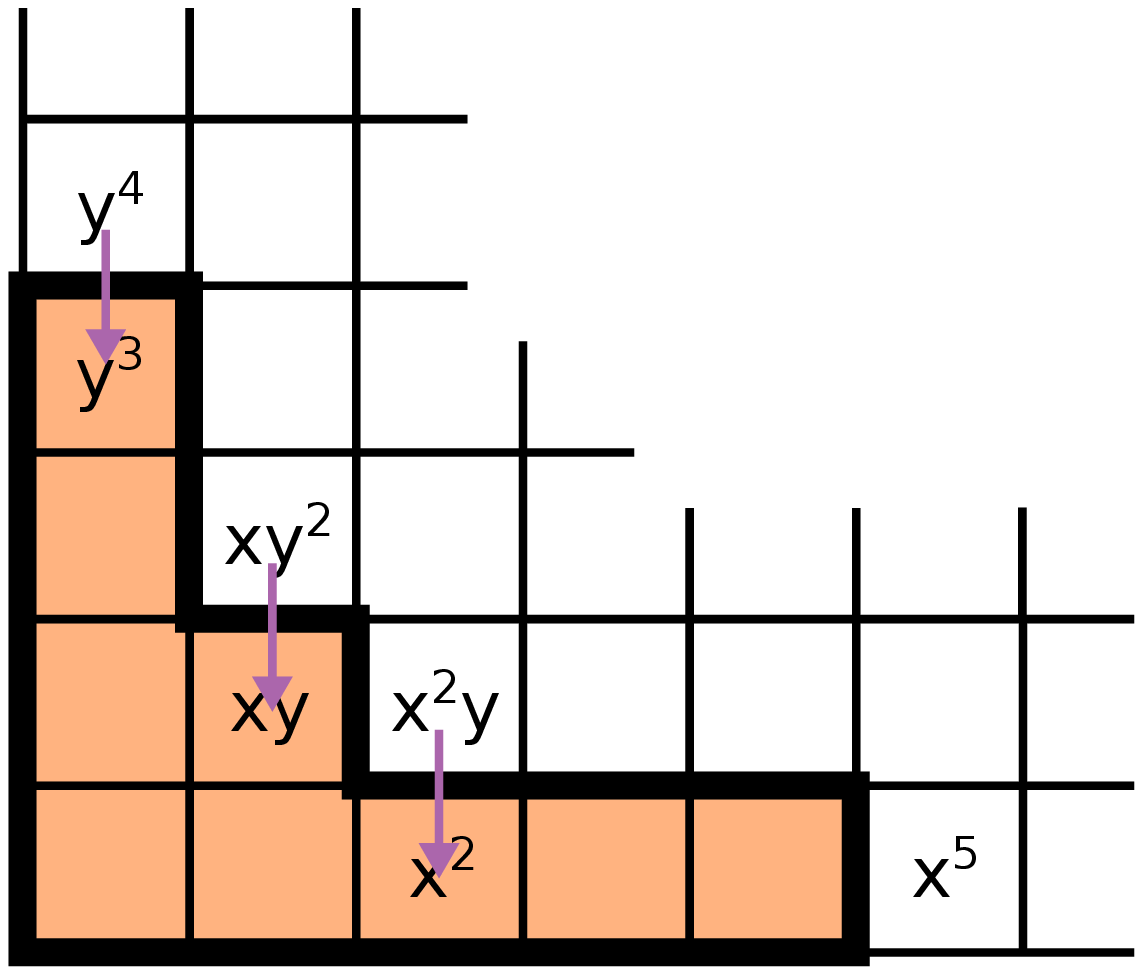}
        \caption*{Figure c. A schematic of $\ddy$ up to scaling.}
    \end{subfigure}
    \hspace{1cm}
    \begin{subfigure}{0.3\textwidth}
        \centering
        \includegraphics[width =\textwidth]{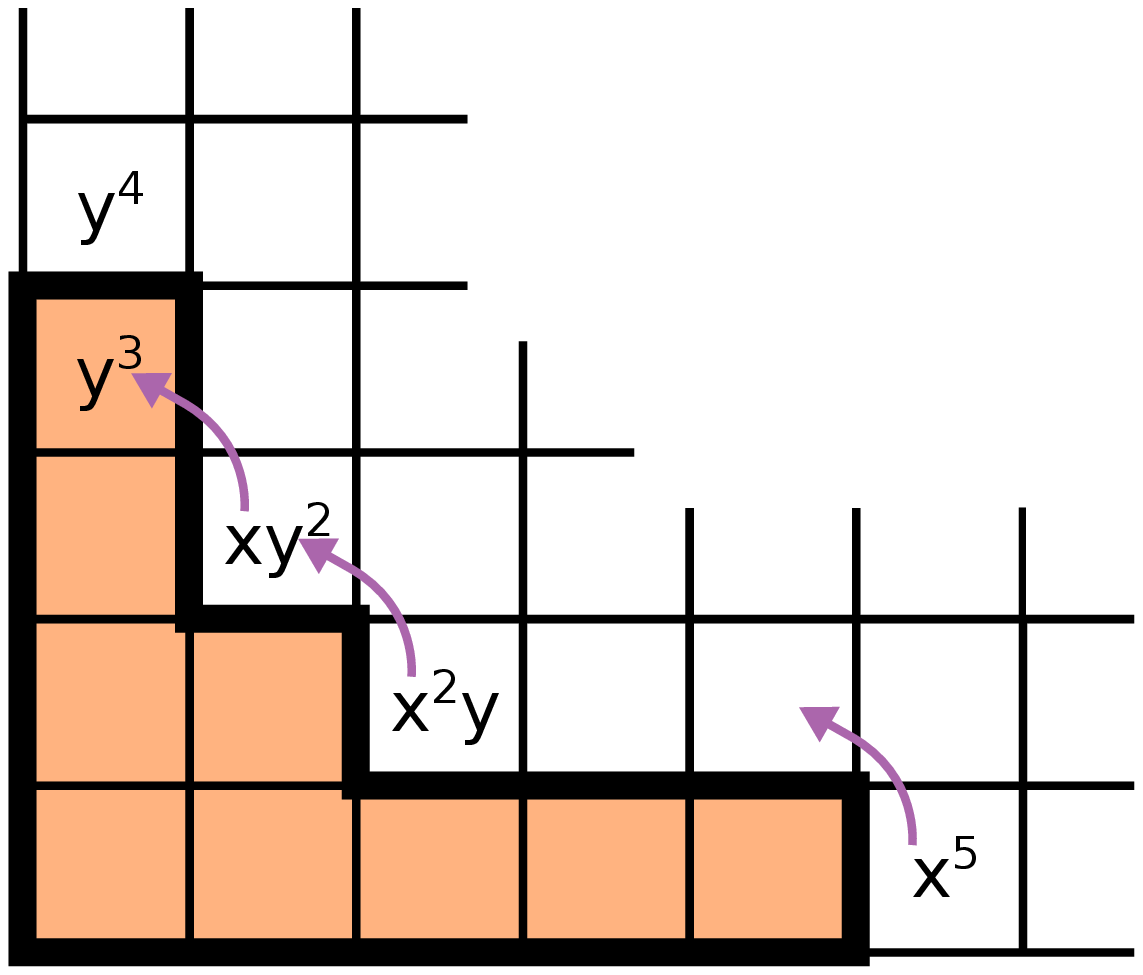}
        \caption*{Figure d. A schematic of $y\ddx$ up to scaling.}
    \end{subfigure}
\end{figure}

Finally we must determine which powers $y^j \ddx$ and $x^i \ddy$ give rise to nonzero homomorphisms. The derivation $y^j \ddx$ acts on $\NN^2$ by shifting to the left 1 and up $j$. This implies when $j \ge \mathrm{max} (\Delta v)$ then $y^j \ddx$ maps all $x^iy^j$ for $(i,j) \in \NN^2 \setminus \mu$ (the ideal) to other points in $\NN^2 \setminus \mu$ (back into the ideal). Thus the associated homomorphism in $\Hom(I_\xi,\CC[\xi])$ is 0. Likewise if $i \ge \mathrm{max}(\Delta h)$ then the derivation $x^i \ddy$ is in the kernel of $\alpha_n|_{[\xi]}$. Lastly it is clear that distinct monomial homomorphisms $y^j\ddx$ for $0\le j < \mathrm{max} (\Delta v)$ and $x^i \ddy$ for $0\le i < \mathrm{max} (\Delta h)$ give rise to linearly independent homomorphisms in $\Hom(I_\xi,\CC[\xi])$, proving $\rk(\alpha_n|_{[\xi]}) = \mathrm{max} (\Delta h) + \mathrm{max} (\Delta v)$.
\end{proof}


\section{Computing the dimension of tangent spaces at monomial subscheme of $\CC^2$}

In this section we prove equality in the \hr{MT}{Theorem A} when $\xi \subset \CC^2$ is cut out by monomials. By Remark 1 and \hr{MC}{Theorem B} it suffices to show:

\begin{proposition} If $\xi \subset \CC^2$ is a monomial subscheme then $$\rk(dh|_{[\xi]}) = \Big( \begin{array}{c}
    \textrm{maximum of horizontal}
  \\
    \textrm{steps of }\mu_\xi
    \end{array} \Big) + \Big( \begin{array}{c}
    \textrm{maximum of vertical}
  \\
    \textrm{steps of }\mu_\xi
    \end{array} \Big).$$
\end{proposition}

Our main computational tool is Haiman's affine charts centered at $[\xi]$. We review without proof the properties of the Haiman chart that we will need and refer the interested reader to \cite[\S 2]{haiman}. If $\mu = \mu_\xi \subset \NN^2$ is the Young diagram associated to $\xi$, then the monomials
$$
\mc{B}_\mu:=\{x^i y^j \in \mu | (i,j)\in \mu\}
$$
give global sections of the rank $n$ tautological bundle $\ocn$ by pulling back and pushing forward. Moreover these sections globally generate $\ocn$ in an open neighborhood $U_\mu \subset \Hnc$ of $[\xi]$.

\begin{definition}
$U_\mu$ is the \emph{Haiman chart centered at $\xi$}.
\end{definition}

\noindent In particular the rank $n$ vector bundle $\ocn$ is trivialized on each $U_\mu$ and $\mc{B}_\mu$ gives an unordered basis for the free sheaf $\ocn|_{U_\mu}$. Note that $U_\mu$ is a $(\CC^*)^2$-invariant neighborhood of $[\xi]$ that consists of:
$$
U_\mu = \Big\{ [\chi] \in \Hnc \Big| \begin{array}{l}
    \CC[x,y]/I_\chi \textrm{ is spanned as a } \CC\textrm{-vector}
  \\
    \textrm{space by monomials in }\mc{B}_\mu
    \end{array}\Big\}.
$$
Indeed $U_\mu$ is affine \cite[Prop. 2.2]{haiman} and the ring of functions on $U_\mu$ is generated by functions $c^{r,s}_{i,j}$ (with $(r,s)$ and $(i,j)\in\NN^2$) whose value $c^{r,s}_{i,j}([\chi])$ at $[\chi] \in \Hnc$ is defined by:
\begin{equation*}
\label{eq_1}
x^ry^s = \sum_{(i,j) \in \mu} c^{r,s}_{i,j}([\chi]) x^iy^j \mod I_\chi. 
\end{equation*}
It is convenient to depict $c^{r,s}_{i,j}$ by an arrow pointing from $(r,s)$ to $(i,j)$.

Denoting the maximal ideal $\fm_{[\xi]} \subset \Oc_{U_\mu}$ of $[\xi] \in U_\mu$, then the cotangent space $\fm_{[\xi]}/\fm_{[\xi]}^2$ is generated by classes of functions $c^{r,s}_{i,j}$ corresponding to arrows with heads in $\mu$ and tails in $\NN^2\setminus \mu$. We now state the key \emph{Haiman relations} for these arrows modulo $\fm_{[\xi]}^2$:
\begin{itemize}
\item[HR1] see \cite[eqn. 2.18]{haiman}. Translating an arrow horizontally or vertically does not change the class it represents modulo $\fm_{[\xi]}^2$ provided the head remains in $\mu$ and the tail remains outside of $\mu$.
\item[HR2] see \cite[eqn. 2.18]{haiman}. If an arrow can be translated so that its head crosses the $x$-axis or the $y$-axis and the tail remains in $\NN^2 \setminus \mu$, then it represents a class that vanishes modulo $\fm_{[\xi]}^2$.
\item[HR3] see \cite[p. 211]{haiman}. Any strictly southwest pointing arrow vanishes modulo $\fm_{[\xi]}^2$.
\end{itemize}
The set of equivalence classes of nonvanishing arrows under the Haiman relations forms a basis for the cotangent space (see Figure e and Figure f for examples of these relations).

\begin{figure}[H]
    \centering
    \begin{subfigure}{0.4\textwidth}
        \centering
        \includegraphics[width=.6\textwidth]{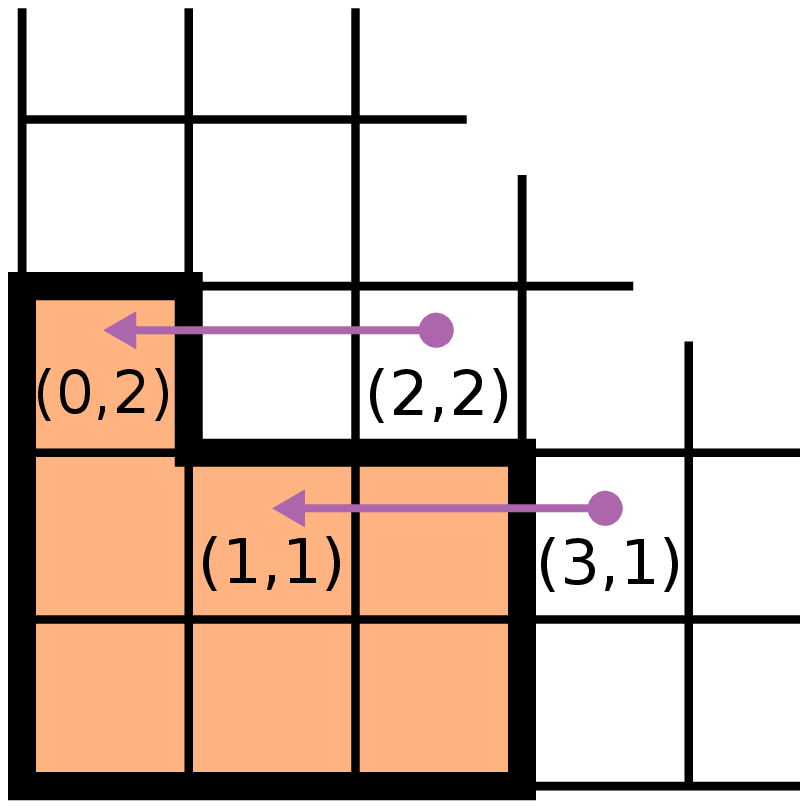}
        \caption*{Figure e. In $\fm_{[\xi]}/\fm_{[\xi]}^2$ we have $c^{22}_{02}\ne c^{31}_{11}$ although the arrows have the same slope.}
    \end{subfigure}
    \hspace{1.5cm}
    \begin{subfigure}{0.45\textwidth}
        \centering
        \includegraphics[width=.6\textwidth]{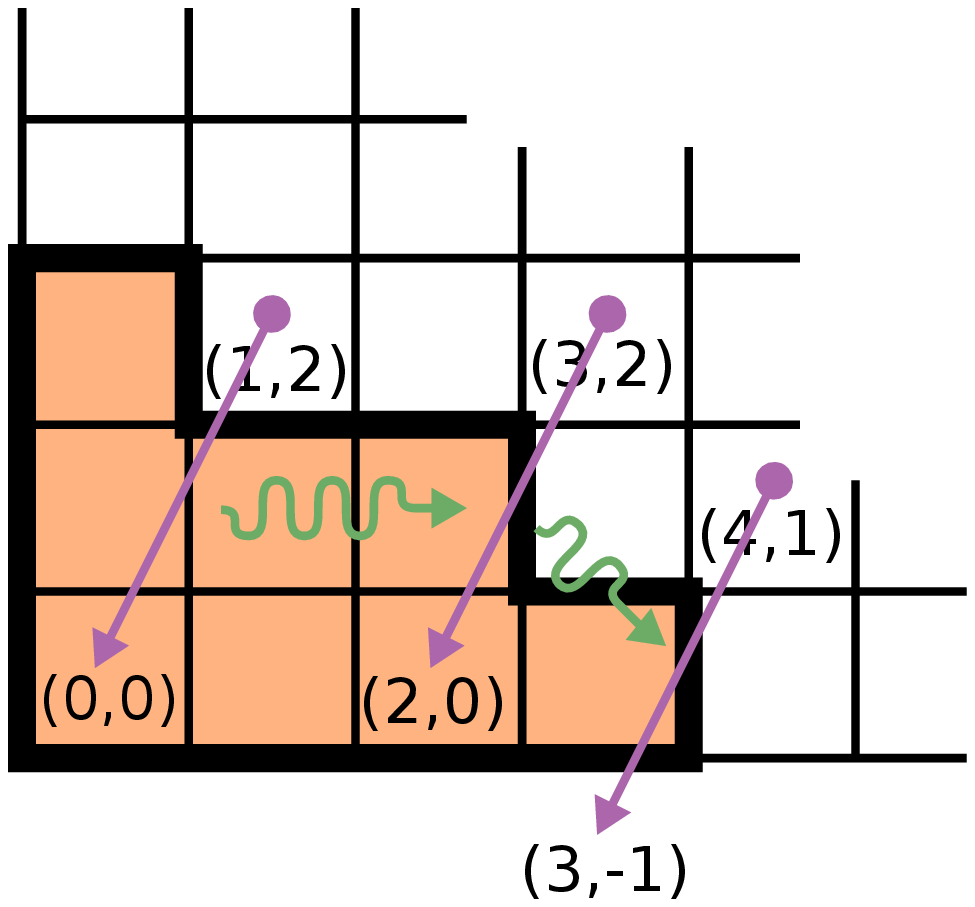}
        \caption*{Figure f. We have the Haiman relations: $c^{12}_{00} = c^{32}_{20} = c^{41}_{3,-1} = 0$ in $\fm_{[\xi]}/\fm_{[\xi]}^2$, verifying HR3.}
    \end{subfigure}
\end{figure}

Let $R = \mb{C}[x_1,\ldots,x_n,y_1,\ldots,y_n]^{\SymG}$ be the coordinate ring of $\Snc = (\CC^2)^n/\SymG$. This ring is generated by the \emph{polarized power sums} \cite{Weyl}:
$$
p_{r,s} = \sum_{i = 1}^n x_i^ry_i^s.
$$
We can describe the pullback of the functions $p_{r,s}$ along $h$ \cite[p. 208]{haiman} as:
$$
h^*(p_{r,s}) = \Tr(x^ry^s:\ocn \ra \ocn)
$$
where $x^ry^s$ is viewed as an endomorphism of $\CC[x,y]/I_\xi$ for $[\xi] \in \Hnc$. Thus $dh|_{[\xi]}$ is the  map on cotangent spaces induced by $h^*$ and its image is spanned by the classes $\Tr(x^ry^s) \mod \fm_{[\xi]}^2$. 

\begin{proof}[Proof of Proposition 6] 

We need to compute the derivative of $\Tr(x^ry^s)$ in $\fm_{[\xi]}/\fm_{[\xi]}^2$. For all $[\chi] \in U_\mu$ we can write $x^r y^s \in \operatorname{End}(\CC[x,y]/I_\chi)$ as a matrix using the basis $\mc{B}_\mu$. Thus we compute the trace:
$$
\Tr(x^ry^s) = \sum_{(h,k) \in \mu} c^{r + h, s+k}_{h,k}
$$ 
as an element of $H^0(U_\mu,\Oc_{U_\mu})$. By the discussion proceeding the proof, the image of $dh|_{[\xi]}$ is generated by 
$$
d(h^*(p_{r,s})) = d \Tr(x^ry^s) \equiv  \sum_{(h,k) \in \mu} c^{r + h, s+k}_{h,k} \mod \fm_{[\xi]}^2.
$$

Using the description of the cotangent space as linear combinations of equivalence classes of arrows on the Young diagram $\mu$, $d(h^*(p_{r,s}))$ is a sum of arrows of slope $s/r$. Whenever both $s$ and $r$ are nonzero, then these arrows are pointing southwest and so by (HR3) they vanish modulo $\fm_{[\xi]}^2$. 

\begin{figure}
    \centering
    \begin{subfigure}{0.45\textwidth}
        \centering
        \includegraphics[width=.85\textwidth]{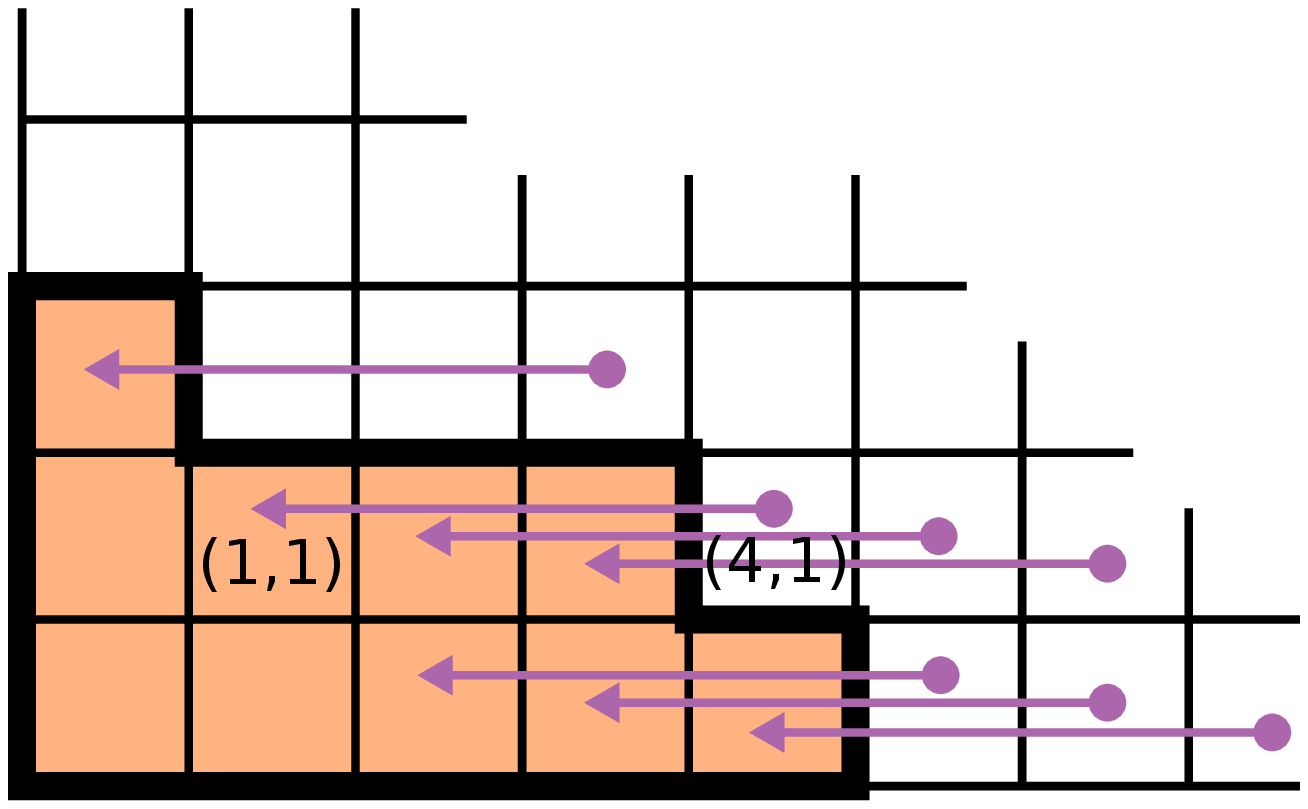}
        \caption*{Figure g. These arrows depict $d(h^*(p_{3,0}))$ modulo $\fm_{[\xi]}^2$. Applying (HR1) we have $d(h^*(p_{3,0}))=6c^{41}_{11}\ne 0$.}
    \end{subfigure}
    \hspace{1cm}
    \begin{subfigure}{0.45\textwidth}
        \centering
        \includegraphics[width=.9\textwidth]{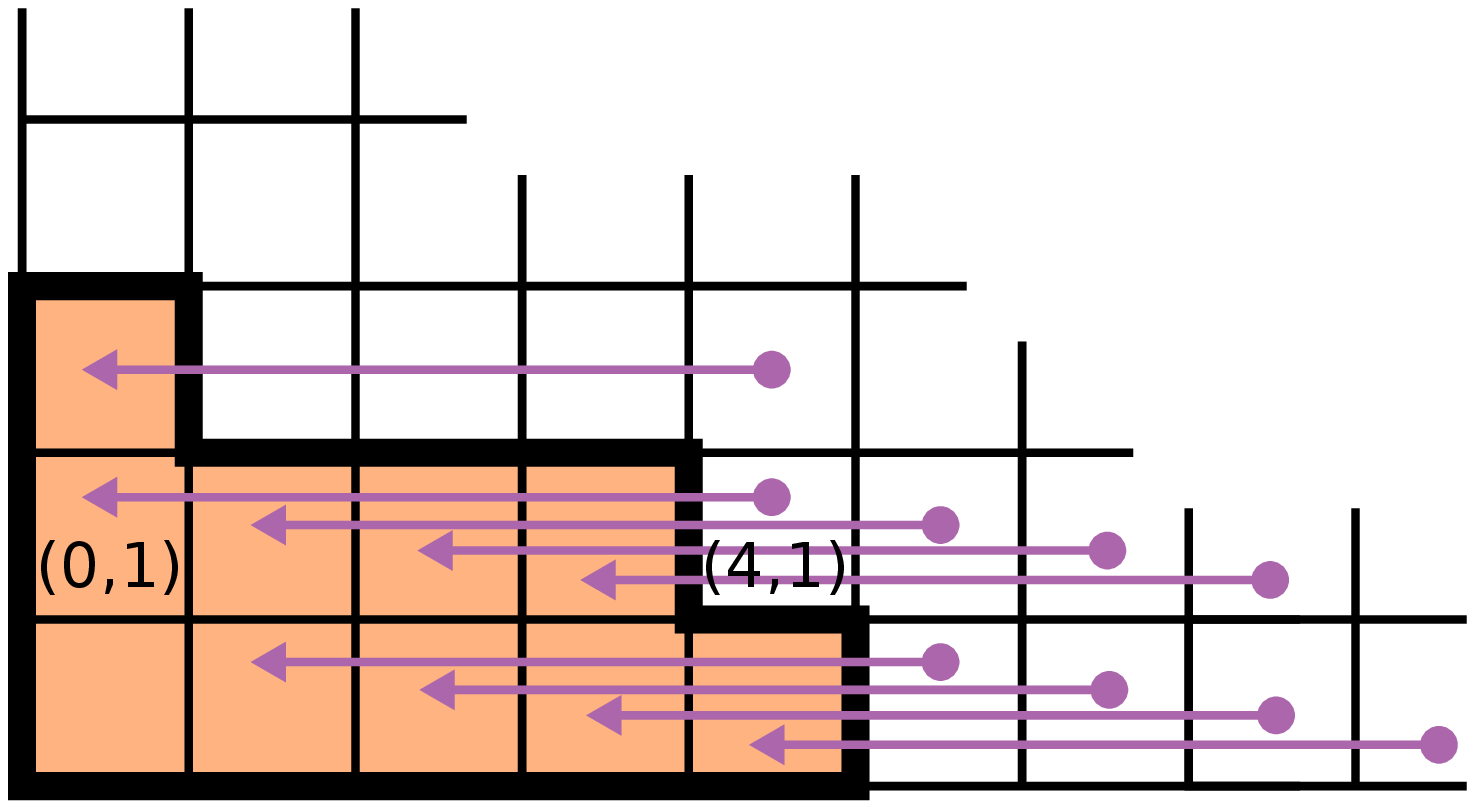}
        \caption*{Figure h. These arrows depict $d(h^*(p_{4,0}))$ modulo $\fm_{[\xi]}^2$. By applying (HR1) and (HR2) to shift up and to the left we see $d(h^*(p_{4,0}))=0$.}
    \end{subfigure}
\end{figure}

When $s = 0$, $d(h^*(p_{r,0}))$ is a sum of horizontal arrows of length $r$. If $r > \max(\Delta h)$, then by (HR1) we can slide each horizontal arrow up and to the right until the head of the arrow leaves the first quadrant (see Figures g and h). Therefore by (HR2), $d(h^*(p_{r,0})) = 0 \mod \fm_{[\xi]}^2$. For $1 \le r \le \max(\Delta h)$, we get that at least one of these arrows is nonzero since we cannot slide any arrow of length $r$ past the max horizontal jump in the diagram while still keeping the head in $\mu$. By the same argument we see that $d(h^*(p_{0,s}))$ is a sum of vertical arrows of length $s$ and is nonzero if and only if $1 \le s \le \max(\Delta v)$. 

Since translation preserves both the length and direction of an arrow, the set 
$$
\{d(h^*(p_{r,0})) \ : \ 1 \le r \le \max(\Delta h)\} \cup \{d(h^*(p_{0,s})) \ : \ 1 \le s \le \max(\Delta v)\}
$$ 
is a linearly independent set of size $\max(\Delta v) + \max(\Delta h)$ which generates $\im(dh|_{[\xi]}) \subset \fm_{[\xi]}/\fm_{[\xi]}^2$ and it follows that
$\rk(dh|_{[\xi]}) = \max(\Delta h) + \max(\Delta v) = \rk(\alpha_n|_{[\xi]}).$
\end{proof}

\begin{remark} Theorem A gives a lower bound on the tangent space dimension of $P_n$. On the other hand, we can obtain upper bounds by taking torus degenerations. In particular, $$2n - \rk(\alpha_n|_{[\xi]}) \le \dim T_{[\xi]}P_n \le \min\{\dim(T_{[\chi]}P_n) \ : [\xi] \text{ degenerates to } [\chi]\}
$$ where $I_\chi$ is a monomial ideal.
\end{remark}

\begin{remark} In fact Proposition $6$ holds more generally for \textit{formally monomial} subschemes, that is, $\xi$ such that there exist formal coordinates around $0 \in \CC^2$ for which $I_\xi$ is a monomial ideal. 

\end{remark}

Recall that a subscheme $\xi \subset \CC^2$ is \textit{curvilinear} if it is contained in a smooth curve. 

\begin{corollary}
Let $\xi$ be a monomial subscheme. Then $[\xi] \in P_n$ is a smooth point if and only if $\xi$ is a curvilinear subscheme.
\end{corollary} 

\begin{proof} A monomial subscheme $\xi$ is curvilinear if and only if $\mu$ is either a single row or a single column of $n$ blocks. In each of these cases one easily sees that $\max(\Delta h) + \max(\Delta v) = n+1$ so that $\dim T_{[\xi]}P_n = n - 1$ is the dimension of $P_n$. 

Suppose $\xi$ is not curvilinear. Let $a$ and $b$ be the length of the largest row and largest column of $\mu$ respectively. Then $a + b \le n + 1$ with equality if and only if $\mu$ is a hook. If $\mu$ is a hook, then $\mu$ must have horizontal step sequence $\Delta h = (1,\alpha)$ and vertical step sequence $\Delta v = (\beta,1)$. Since the horizontal steps add up to $a$ and the vertical steps add up to $b$, it follows that 
$$
\max(\Delta h) + \max(\Delta v) = \alpha + \beta < \alpha + \beta + 2 = a + b = n+1.
$$
On the other hand, if $\mu$ is not a hook, then $\max(\Delta h) + \max(\Delta v) \le a + b < n + 1$. In either case, $\rk(\alpha_n|_{[\xi]}) < n + 1$ so that $\dim T_{[\xi]}P_n = 2n - \rk(\alpha_n|_{[\xi]}) > n - 1 = \dim P_n$.

\end{proof}

It also follows from Theorem A that the maximally singular points of $P_n$ are precisely the $k^{th}$ order neighborhoods of the origin.

\begin{corollary} If $\dim T_{[\xi]}P_n = 2n - 2$ then $I_{\xi} = \fm^k$ where $\fm$ is the maximal ideal of $0 \in \CC^2$. \end{corollary} 

\begin{proof} We have an action of $(\CC^*)^2$ on $\Hnc$ with fixed points corresponding to monomial subschemes so that $P_n$ is invariant. Consider a 1-parameter subgroup $\sigma : \CC^* \to (\CC^*)^2$ which acts on $\Hnc$ with the same fixed points. The limits
$$
\lim_{t\to0}\sigma(t)\cdot[\xi] = [\chi] \enspace \text{ and } \enspace  \lim_{t\to\infty}\sigma(t)\cdot[\xi] = [\zeta],
$$
exist by properness of $P_n$ and they are monomial subschemes. Then by Remark 8 we have:
$$
2n-2 =\dim T_{[\xi]}P_n \le \dim T_{[\chi]}P_n \enspace \text{ and } \enspace 2n-2 =\dim T_{[\xi]}P_n \le \dim T_{[\zeta]}P_n.
$$
Thus $\rk(\alpha_n|_{[\chi]}) = \rk(\alpha_n|_{[\zeta]}) \le 2$. This is only possible if the Young diagrams $\mu_\chi$ and $\mu_\zeta$ are staircases, that is $I_\chi = I_\zeta = \fm^k$ and $\zeta = \chi$. Thus both the degenerations occur in the Haiman chart $U_\chi$ so degeneration gives a map from $\PP^1$ to $U_\chi$. Since $U_\chi$ is affine, the map is constant and $I_\xi = I_\chi = \fm^k$.
\end{proof}

\begin{corollary} If $xy \in I_{\xi}$ then $\dim T_{[\xi]}P_n \le n+1$.
\end{corollary}

\begin{proof}
The fact that $xy \in I_\xi$ implies the only Haiman charts that contain $[\xi]$ correspond to Young diagrams which are hooks. Therefore $\xi$ can only degenerate to monomial schemes with hooks for Young diagrams. An easy computation using the \hr{MT}{Theorem A} and the \hr{MC}{Theorem B} shows that if $\chi\subset \CC^2$ is a monomial subscheme with a hook for a Young diagram then $\dim T_{[\chi]}P_n = n-1$ or $n+1$. Then we are done by Remark 8.
\end{proof}

\bibliographystyle{amsalpha} 
\bibliography{tangent}

\end{document}